\documentclass[11pt]{article}
\usepackage{amsmath,amstext,amssymb,amsopn,amsthm,graphicx}
\theoremstyle{plain}
\newtheorem{thm}{Theorem}[section]

\newtheorem{lem}[thm]{Lemma}
\newtheorem{prop}[thm]{Proposition}
\newtheorem{cor}[thm]{Corollary}

\theoremstyle{definition}

\newtheorem{rem}[thm]{Remark}
\newtheorem*{rem*}{Remark}

\newcommand{\Ai}{{\rm Ai}}
\newcommand{\Bi}{{\rm Bi}}
\newcommand{\sgn}{{\rm sgn}}
\newcommand{\ex}{\mathbb{E}}
\newcommand{\pr}{P}

\newcommand{\cF}{\mathcal{F}}

\newcommand{\cN}{\mathcal{N}}

\providecommand{\pro}[1]{(#1_t)_{t \geq 0}}

\providecommand{\seq}[1]{(#1_n)_{n\in \mathbb{N}}}

\DeclareMathOperator{\Spec}{Spec}
\DeclareMathOperator{\Dom}{Dom}

\makeatletter
\@addtoreset{equation}{section}
\makeatother

\newcommand{\R}{\mathbb{R}}

\newcommand{\E}{\mathbb{E}}

\renewcommand{\leq}{\leqslant}
\renewcommand{\geq}{\geqslant}

\renewcommand{\leq}{\leqslant}
\renewcommand{\geq}{\geqslant}

\def\({\left(}
\def\){\right)}
\def\[{\left[}
\def\]{\right]}
\def\<{\langle}
\def\>{\rangle}

\date{}

\title{\Large \sc{Spectral Properties of the Massless Relativistic \\ Harmonic Oscillator}
\footnotetext{2010 MS Classification: Primary 47G30; Secondary 60G52, 35P05. The work of the second author was supported in part by the Polish Ministry of Science and Higher Education grant no. N N201 373136.}
\author{
\small J\'ozsef L{\"{o}}rinczi\\[0.1cm]
{\it \small School of Mathematics, Loughborough University} \\[-0.7ex]
{\it \small Loughborough LE11 3TU, United Kingdom} \\[-0.7ex]
{\small {\tt  J.Lorinczi@lboro.ac.uk}}\\[0.5cm]
\small Jacek Ma{\l}ecki\\[0.1cm]
{\it \small Institute of Mathematics and Computer Science} \\[-0.7ex]
{\it \small  Wroc{\l}aw University of Technology} \\[-0.7ex]
{\it \small Wyb. Wyspia{\'n}skiego 27, 50-370 Wroc{\l}aw, Poland} \\[-0.7ex]
{\small  {\tt jacek.malecki@pwr.wroc.pl}} }
}

\begin{document}
\maketitle

\vspace{3cm}
\begin{abstract}
\noindent
The spectral properties of the pseudo-differential operator $(-d^2/dx^2)^{1/2}+x^2$ are analyzed by a combination
of functional integration methods and direct analysis. We obtain a representation of its eigenvalues and
eigenfunctions, prove precise asymptotic formulae, and establish various analytic properties. We also derive trace
asymptotics and heat kernel estimates.
\end{abstract}

\newpage
\section{Introduction}
Stochastic methods based on functional integration applied to the study of properties of pseudo-differential
operators and related semigroups offer a powerful alternative to the techniques of analysis \cite{B,J}. Typical
problems addressed include spectral properties of the operator, heat kernel estimates, $L^p$-boundedness,
ultracontractivity properties, and the decay of the eigenfunctions.

In the paper \cite{KL} we study analytic properties of evolution semigroups generated by fractional Schr\"odinger
operators
$$
H_\alpha = (-\Delta)^{\alpha/2} + V, \quad 0 < \alpha < 2
$$
with (fractional) Kato-class potentials $V$. As $\alpha \neq 2$ these operators generate non-Gaussian
$\alpha$-stable processes running under the potential $V$. These are L\'evy processes with paths having jump
discontinuities. The well-known case $\alpha=2$ corresponds to standard Schr\"odinger operators generating
Brownian motion in the presence of $V$ (which is an It\^o diffusion under extra conditions on the potential).
The properties of Schr\"odinger operators and fractional Schr\"odinger operators in many aspects markedly differ.
One sharp contrast appears in the decay properties of their ground state (first eigenfunction). Specifically,
the ground state of a Schr\"odinger operator with pinning potential $V(x) \to \infty$ as $|x|\to\infty$ decays
(super)exponentially while it decays only polynomially in the case of fractional Schr\"odinger operators. This
difference is due to the heavy tails of stable processes, as opposed to the Gaussian tail of Brownian motion.
Another remarkable difference is that, roughly, the Schr\"odinger semigroups $e^{-t(-\Delta+V)}$ are intrinsically
ultracontractive for super-quadratically increasing potentials, while this property holds for fractional
Schr\"odinger semigroups $e^{-t\left((-\Delta)^{\alpha/2}+V\right)}$ already for potentials increasing faster
than logarithmically.

Our aim in the present paper is to focus on one single pseudo-differential operator and derive fine details on
its spectrum and eigenfunctions by using a combination of functional integration and hands-on analytic methods.
We will consider the operator
$$
H = \sqrt{-\frac{d^2}{dx^2}}+x^2
$$
and obtain various formulae and estimates on its kernel, eigenfunctions, and spectrum. The interest in this
particular choice is twofold. One is that in order to further develop the more general theory it is important
to have cases of reference with as detailed information as possible. The results we obtain below are indeed
more refined than the general methods using either pseudo-differential calculus or functional integration
provide, and we view this paper as complementary to the more general results in \cite{HIL,LHB}. A second
motivation is that there is much controversy in the physics literature (see, for instance, \cite{L,DX}) about
claimed solutions of fractional Schr\"odinger equations. Due to non-locality of these operators such equations
are more delicate than usual Schr\"odinger equations and an appropriate rigorous mathematical treatment is
necessary. The operator we consider describes the massless (semi-) relativistic quantum harmonic oscillator
studied in physics.

Our main results are as follows. First we derive a functional integral representation which allows to define
$H$ as a self-adjoint operator. From the results of \cite{KK} it follows that the first eigenfunction (ground
state) $\varphi_1$ of $H$ is bounded both from below and above by $x^{-4}$ with suitable prefactors, moreover,
for all other eigenfunctions $|\varphi_n(x)|\leq \mbox{const}\,\varphi_1(x)$ holds. Due to the special choice
of the potential, by using special functions we improve this result to a detailed asymptotic expansion of
eigenfunctions, in particular, tighten the order of magnitude on the bounds
(Theorem \ref{Eigenfunctions:asymptotics} below). Secondly, in \cite{KKMS} it was proven that the eigenfunctions
are uniformly bounded for the case of the Cauchy process run in an interval only. We prove uniform boundedness of
all eigenfunctions for $H$ on $\R$ (Theorem \ref{UB:thm}), and show that the set of zeroes of each eigenfunction
is finite (Corollary \ref{zeroes}). Also, we discuss the shape of the ground state (Theorem \ref{shapegs}).
On the eigenvalues our main results are the precise asymptotic expansions in Corollary \ref{Eigenvalues:asymp}
resulting from Theorem \ref{eigenval}, in which we prove that the eigenvalues are simple. We also obtain a
spectral gap estimate (Corollary \ref{specgap}), and derive the trace asymptotics in Theorem \ref{trestim}. We
give a heat kernel estimate in Theorem \ref{heatkernel:estimates:1}.

The plan of the paper is the following. In Section 2 we derive a Feynman-Kac-type formula for a class of operators
covering our case and then particularize to our chosen operator. We use the functional representation to define
the operator as a self-adjoint operator. In Section 3 we show that the Fourier transform of the eigenfunctions
satisfy the Airy equation under
appropriate boundary conditions. This allows us to identify the spectrum of the fractional harmonic oscillator
operator and derive some asymptotic formulae. Furthermore, here we present the main results as discussed above.
In a short Section 4 we provide an appendix of the used facts on Airy functions.

\section{Functional integral representation}
Recall that the linear operator with domain $H^\alpha(\R^d)= \{ f \in L^2(\R^d): \, |k|^{\alpha}\hat f\in L^2(\R^d)\}$,
$0 < \alpha < 2$, $d \geq 1$, defined by $\widehat{(-\Delta)^{\alpha/2} f}(k) = |k|^\alpha \hat f(k)$, is the
\emph{fractional Laplacian} of order $\alpha$. It is essentially self-adjoint on $C_0^\infty(\R^d)$, and its spectrum
is $\Spec ((-\Delta)^{\alpha/2}) =\Spec_{\rm{ess}}((-\Delta)^{\alpha/2})=[0,\infty)$.

Let $V: \R^d \to \R$ be a bounded Borel measurable function. We call $(-\Delta)^{\alpha/2} + V$, $0 < \alpha < 2$,
a \emph{fractional Schr\"odinger operator} with potential $V$, where $V$ acts as a multiplication operator. Since
$V$ is a bounded function, the operator $(-\Delta)^{\alpha/2} + V$ is self-adjoint on $\Dom((-\Delta)^{\alpha/2})$
defined as a sum of two self-adjoint operators. Therefore $\Spec ((-\Delta)^{\alpha/2} + V) \subset [0,\infty)$.

Let $(\Omega_X, {\cF}_X, \pr_{X})$ be a probability space and $\pro X$ a real valued symmetric
$\alpha$-stable process on it, with $0 < \alpha < 2$. $\pro X$ is a non-Gaussian L\'evy process, in
particular it has independent and stationary increments. We use the notations $\pr^x$ and $\ex^x$,
respectively, for the distribution and the expected value of the process starting in $x \in \R$ at
time $t=0$; for simplicity we do not indicate the measure in subscript (while we do when have any
other measure or process). The characteristic function of $\pro X$ is given by
\begin{equation}
\ex^0[e^{i \xi X_t}] = e^{-t|\xi|^\alpha}, \quad \xi \in \R, \, t > 0.
\label{charstable}
\end{equation}
As a L\'evy process, $\pro X$ has a version with paths in $D([0,\infty); \R^d)$, the space of c\`adl\`ag
functions (i.e., right continuous functions with existing left limits).

Recall that a subordinator $\pro S$ on a given probability space $(\Omega_{S}, \cF_{S}, \pr_{S})$ is
an almost surely non-decreasing $[0,\infty)$-valued L\'evy process starting at $0$. An example is the
$(\alpha/2)$-stable subordinator $\pro S$ uniquely determined by its Laplace transform
\begin{equation}
\label{df:Ltr}
\ex_{\pr_S}^0[e^{-\lambda S_t}] = e^{-t\lambda^{\alpha/2}}, \quad t\geq 0, \, \lambda \geq 0.
\end{equation}

Consider standard Brownian motion $\pro B$ on a given probability space $(\Omega_W, \cF_W, \pr_{W})$,
where $\pr_W$ is Wiener measure. Clearly,
\begin{equation}
\ex^0_{\pr_W}[e^{i \xi B_t}] = e^{-t|\xi|^2}, \quad \xi \in \R, \, t > 0.
\label{charBM}
\end{equation}
It is a standard fact that any symmetric $\alpha$-stable process $\pro X$
can be obtained as a random time change of Brownian motion where this random time process is an
$(\alpha/2)$-stable subordinator $\pro S$. It is convenient to consider the processes $\pro B$ and
$\pro S$ on two different probability spaces $(\Omega_W, \cF_W, \pr_{W})$ and $(\Omega_{S}, \cF_{S},
\pr_{S})$. Then the process $\pro X$ can be obtained in terms of subordinate Brownian motion with
respect to the $(\alpha/2)$-stable subordinator:
$$
X_t : \Omega_{\pr_{W}} \times \Omega_{\pr_{S}} \ni (\omega, \tau) \longmapsto B_{S_t(\tau)}(\omega)
:= X_t(\omega,\tau).
$$
This can also be seen by the composition of the characteristic exponent (\ref{charBM}) with the
Laplace exponent (\ref{df:Ltr}) which gives (\ref{charstable}). Furthermore, $\pr$ can then be
identified as the image measure of this process on $D([0,\infty); \R^d)$ such that
$$
\pr^x(X_t \in A) = (\pr^x_{W} \times \pr^0_{S}) (B_{S_t} \in A)
$$
holds for all Borel sets $A \subset \R^d$.

In what follows we take the case $d=1$, $\alpha=1$, i.e., the operators $\sqrt{-\frac{d^2}{dx^2}}$,
$V(x) = x^2$ so that we will consider the fractional Schr\"odinger operator
\begin{equation}
H := \left(-\frac{d^2}{dx^2}\right)^{1/2} + x^2.
\label{H}
\end{equation}
The symmetric 1-stable process $\pro X$ is also known as \emph{Cauchy process} whose one-dimensional
distributions are given explicitly by
\begin{equation*}
P^x(X_t\in dy) = \frac{1}{\pi} \frac{t}{t^2+(x-y)^2}dy\/,\quad x\in\R\/.
\end{equation*}
Our main concern in this paper is to study the spectral properties of $H$ by using functional integration
methods.

First we show how $H$ relates with the Cauchy process. We have the following Feynman-Kac-type formula, which
we state in $d$ dimensions and a class of $V$ containing our special case (see \cite{HIL,KL} for more general
pseudo-differential operators).
\begin{thm}
Let $V \in L^\infty(\R^d)$ and $\pro X$ be a $d$-dimensional Cauchy process. We have
\begin{equation}
\label{eq:FKFfrSch}
(f, e^{-t\left(\sqrt{-\Delta}+V\right)}g)= \int_{\R^d} dx
\ex^x \left[\overline{f(X_0)}g(X_t) e^{-\int_0^t V(X_s)ds}\right].
\end{equation}
\label{FK}
\end{thm}
\begin{proof}
We divide the proof into four steps.

\medskip
\noindent
\emph{(Step 1)} Suppose $V \equiv 0$. Our first claim is
\begin{align}
\label{eq:st1}
(f,e^{-t\left(\sqrt{-\Delta}\right)}g)=\int_{\R^d} dx \ex^x \left[\overline{f(X_0)}g(X_t)\right].
\end{align}
We regard the process $\pro X$ as the composition of Brownian motion $\pro B$ and the 1/2-stable
subordinator $\pro S$ as explained above. Let $E_\lambda$ denote the spectral projection of the
self-adjoint operator $-\Delta \geq 0$. Then by using \eqref{df:Ltr} and the usual Feynman-Kac
formula for $e^{t\Delta}$ we have
\begin{align*}
\left(f, e^{-t\sqrt{-\Delta}}g \right)
& = \int_0^\infty e^{-t\sqrt\lambda} d(f,E_\lambda g)
= \int_0^\infty \ex^0_{\pr_S}\left[e^{-\lambda S_t}\right]  d(f,E_\lambda g) \\
& = \ex^0_{\pr_S}\left[\int_0^\infty e^{-S_t\lambda} d(f,E_\lambda g) \right]
= \ex^0_{\pr_S}\left[(f,e^{-S_t (-\Delta)} g) \right] \\
& =  \ex^{0}_{\pr_S}\left[\int_{\R^d} dx\ex_{\pr_W}^x[\overline{f(B_0)} g(B_{S_t})]\right]
= \int_{\R^d} dx \ex^x[\overline{f(X_0)} g(X_t)],
\end{align*}
thus \eqref{eq:st1} follows.

\medskip
\noindent
\emph{(Step 2)}
Let $0 = t_0 < t_1 < ...< t_n$, $f_0, f_n \in L^2(\R^d)$ and assume that $f_j \in L^\infty(\R^d)$,
for $j = 1,2,...,n-1$. We claim that
\begin{align}
\label{eq:st2}
\left(f_0, \prod_{j=1}^n e^{-(t_j - t_{j-1}) \sqrt{-\Delta}}f_j\right) =
\int_{\R^d} dx \ex^x \left[\overline{f(X_0)}\prod_{j=1}^n f_j(X_{t_j}) \right].
\end{align}
For simplifying the notation put $s_j = t_j-t_{j-1}$, for any $j=1,...,n$ and
$$
g_i = f_i\left(\prod_{j=i+1}^n e^{-s_j \sqrt{-\Delta}} f_j\right),
\quad j=1,..., n-1, \, g_n = f_n.
$$
Notice that $g_j = f_j e^{-s_{j+1} \sqrt{-\Delta}} g_{j+1}$. By \eqref{eq:st1} the left hand
side of \eqref{eq:st2} can be represented as
$$
\int_{\R^d} dx \ex^x \left[\overline{f(X_0)}g_1(X_{s_1}) \right]= \int_{\R^d} dx
\overline{f(x)}\ex^x \left[g_1(X_{s_1}) \right].
$$
Using \eqref{eq:st1} again, we obtain
\begin{align*}
\ex^x \left[g_j(X_{s_j}) \right] & = \int_{\R^d} p(s_j,y-x) g_j(y) dy
 = \int_{\R^d} p(s_j,y-x) f_j(y) e^{-s_{j+1}\sqrt{-\Delta}}g_{j+1}(y) dy\\
& = \int_{\R^d} \ex^y \left[p(s_j,X_0-x) f_j(X_0) g_{j+1}(X_{s_{j+1}})\right] dy \\
& = \int_{\R^d} p(s_j,y-x) f_j(y) \ex^y \left[g_{j+1}(X_{s_{j+1}})\right] dy
 = \ex^x \left[f_j(X_{s_j}) \ex^{X_{s_j}} \left[g_{j+1}(X_{s_{j+1}})\right]\right],
\end{align*}
for $j=1,...,n-1$. The above equalities yield
\begin{align*}
& \left(f_0, \prod_{j=1}^n e^{-s_j \sqrt{-\Delta}}f_j\right) =
\int_{\R^d} dx  \ex^x\left[\overline{f(X_0)}f_1(X_{s_1}) \times \right.\\
& \qquad \qquad \left. \times \ex^{X_{s_1}}\right.\left[f_2(X_{s_2})\ex^{X_{s_2}}\right.
\left[f_3(X_{s_3})\ex^{X_{s_3}}\right. \left.\left.\left.\left[\ldots \ex^{X_{s_{n-1}}}
\left[f_n(X_{s_n})\right] \ldots \right]\right] \right]\right],
\end{align*}
and \eqref{eq:st2} follows by the Markov property of $\pro X$.

\medskip
\noindent
\emph{(Step 3)} Let now $0 \neq V \in C_b(\R^d)$. We show \eqref{eq:FKFfrSch} for such $V$.
Since $\sqrt{-\Delta}$ is self-adjoint, the Trotter product formula holds:
$$
\left(f,e^{-t(\sqrt{-\Delta} + V)}g\right) =
\lim_{n \to \infty} \left(f,\left(e^{-(t/n)\sqrt{-\Delta}} e^{-(t/n)V}\right)^n g\right).
$$
Combined with (Step 2) it yields
$$
\left(f,e^{-t(\sqrt{-\Delta} + V)}g\right) = \lim_{n \to \infty} \int_{\R^d} dx
\ex^x \left[\overline{f(X_0)} g(X_t) e^{-\sum_{j=1}^n  (t/n)V(X_{tj/n})}\right].
$$
Since each c\`adl\`ag path $s \mapsto \omega(s) = X_s(\omega)$ is continuous in $s \in [0,t]$ except
for at most finite points, we have $\sum_{j=1}^n  (t/n)V(X_{tj/n}) \to \int_0^t V(X_s) ds$ as $n \to
\infty$ in the sense of Riemann integral. Thus \eqref{eq:FKFfrSch} follows for $V \in C_b(\R^d)$.

\medskip
\noindent
\emph{(Step 4)} We make use of the argument in \cite[Theorem 6.2]{S2} to complete the proof.
Suppose that $V \in L^\infty(\R^d)$ and let $V_n = \phi(x/n)(V \ast h_n)$, where $h_n = n^d \phi(nx)$
with $\phi \in C^\infty_0(\R^d)$ such that $0 \leq \phi \leq 1$, $\int \phi(x) dx = 1$ and $\phi(0)
= 1$. Then $V_n \to V$ almost everywhere and $V_n$ are bounded and continuous. Let $\cN$ denote the
set of all $x$ such that $V_n(x)$ does not converge to $V(x)$. Then the measure of $\left\{t \in [0,\infty):
X_t(\omega) \in \cN\right\}$ is zero $\pr^x$-almost surely and $\int_0^t V_n(X_s) ds \to \int_0^t V(X_s)ds$
as $n \to \infty$ $\pr^x$-a.s. Thus
$$
\int_{\R^d} dx \ex^x \left[\overline{f(X_0)} g(X_t) e^{-\int_0^t V_n(X_s) ds }\right] \to \int_{\R^d} dx
\ex^x \left[\overline{f(X_0)} g(X_t) e^{-\int_0^t V(X_s) ds }\right]
$$
as $n \to \infty$. On the other hand, $e^{-t(\sqrt{-\Delta} + V_n)} \to e^{-t(\sqrt{-\Delta}+V)}$ in
strong sense as $n \to \infty$, since $\sqrt{-\Delta} + V_n \to \sqrt{-\Delta} + V$ on the domain
$\Dom(\sqrt{-\Delta})$.
\end{proof}

We can use Theorem \ref{FK} to define $\sqrt{-\Delta}+V$ as a self-adjoint operator. We define the
\emph{Feynman-Kac semigroup}
\begin{equation*}
(T_tf)(x) = \ex^x\[e^{-\int_0^t V(X_s)ds}f(X_t)\]\/, \quad f\in L^2(\R^d)\/, \; x\in\R^d.
\end{equation*}

\begin{thm}
\label{prop:FKS}
Let $V \in L^1_{\rm loc}(\R^d)$ such that $V\geq 0$. Then $\{T_t: t\geq 0\}$ is a strongly continuous symmetric
semigroup. In particular, there exists a self-adjoint operator $K$ bounded from below such that $e^{-tK}
= T_t$.
\end{thm}
\noindent
$K$ can be identified as the self-adjoint operator $\sqrt{-\Delta} + V$ for $0 \leq V \in L^1_{\rm loc}(\R^d)$.
\begin{proof}
Since $0 \leq V \in L^1_{\rm loc}(\R^d)$ there exist $C^{(0)}_V$, $C^{(1)}_V > 0$ such that
\begin{align}
\label{eq:integrab}
\sup_{x \in \R^d} \ex^x[e^{-\int_0^t V(X_s)ds}] \leq e^{C^{(0)}_V + C^{(1)}_V t}.
\end{align}
Therefore we have
\begin{eqnarray*}
\|T_t f\|^2
&\leq &
\int_{\R^d}dx \ex^x \left[e^{-2\int_0^t V(X_s)ds}|f(X_t)|^2\right] \\
&\leq&
C_t \int_{\R^d}dx \ex^x \left[|f(X_t)|^2 \right] \\
&=&
C_t \|e^{(t/2)\Delta}f\|^2 \leq C_t\|f\|^2,
\end{eqnarray*}
with some $C_t > 0$. Thus $T_t $ is a bounded operator from $L^2(\R^d)$ to $L^2(\R^d)$. Similarly as in
Step~2 of the proof of Theorem \ref{FK} it is seen that the semigroup property $T_t T_s=T_{t+s}$ holds
for  $t,s\geq 0$.

To obtain strong continuity of $T_t $ in $t$ it suffices to show weak continuity. Let $f,g \in
C_0^\infty(\R^d)$. Then we have
$$
(f,T_t g)=\int_{\R^d}dx \ex_{\pr_W\times\pr_S}^{x,0}
\left[ \overline{f(B_0)} g(B_{S_t}) e^{-\int_0^t V(B_{S_r}) dr} \right].
$$
Since $S_t(\tau)\rightarrow 0$ as $t\rightarrow 0$ for each $\tau\in \Omega_{\pr_S}$, dominated convergence
gives $(f,T_t g)\rightarrow (f,g)$.

Finally, we check the symmetry property $T_t^\ast =T_t$. Let $\widetilde B_s=\widetilde B_s(\omega,\tau)=
B_{S_t(\tau)-s}(\omega)-B_{S_t(\tau)}(\omega)$. For every $\tau\in \Omega_{\pr_S}$ we have then $\widetilde
B_s\stackrel {\rm d}{=} B_s$ with respect to $\pr_W^x$ ($\stackrel {\rm d}{=}$ denotes that the random variables
are identically distributed). Hence
\begin{eqnarray*}
(f, T_t g)
&=&
\int_{\R^d}dx \overline{f(x)} \ex_{\pr_W\times\pr_S}^{x,0} \left[e^{-\int_0^t V(\widetilde B_{S_r})dr}
g(\widetilde B_{S_t}) \right]\\
&=&
\ex_{\pr_W\times\pr_S}^{0,0} \left[\int_{\R^d}dx \overline{f(x)} e^{-\int_0^t V(x+\widetilde B_{S_r})dr}
g(x+\widetilde B_{S_t}) \right]\\
&=&
\ex_{\pr_W\times\pr_S}^{0,0} \left[\int_{\R^d}dx \overline{f(x-\widetilde B_{S_t})}
e^{-\int_0^t V(x+\widetilde B_{S_r}-\widetilde B_{S_t})dr} g(x) \right].
\end{eqnarray*}
In the second equality we changed the variable $x$ to $x-\widetilde B_{S_t}$. Since $\widetilde B_{S_t}
\stackrel{\rm d}{=}-B_{S_t}$ and $\widetilde B_{S_r}-\widetilde B_{S_t}\stackrel{\rm d}{=}B_{S_t-S_r}$,
we have
\begin{eqnarray*}
(f, T_t g)= \int_{\R^d}dx \ex_{\pr_W\times\pr_S}^{x,0} \left[\overline{f(B_{S_t})}
e^{-\int_0^t V(B_{S_t-S_r})dr} g(x) \right].
\end{eqnarray*}
Moreover, as $S_t-S_r\stackrel{\rm d}{=}S_{t-r}$ for $0\leq r\leq t$, we obtain
\begin{eqnarray*}
(f, T_t g)
&=&
\int_{\R^d}dx \ex_{\pr_W\times\pr_S}^{x,0} \left[\overline{f(B_{S_t})}
e^{-\int_0^t V(B_{S_{t-r}})dr }g(x) \right]\\
&=&
\int_{\R^d}dx \overline{\ex_{\pr_W\times\pr_S}^{x,0} \left[f(B_{S_t})e^{-\int_0^t V(B_{S_r})dr }\right]} g(x)
= (T_t f, g).
\end{eqnarray*}
The existence of a self-adjoint operator $K$ bounded from below such that $T_t = e^{-tK}$ follows now by the
Hille-Yoshida theorem. This completes the proof.
\end{proof}

\section{Eigenvalues and eigenfunctions of $H$}
\subsection{Basic regularity properties}
From now on we consider $H$ defined by (\ref{H}) and the related Feynman-Kac semigroup. $\{T_t: t \geq 0\}$ is
given by an integral kernel, i.e., there exists $u(t,x,y)$ such that
\begin{equation*}
(T_tf)(x) = \int_{\R} u(t,x,y)f(y)dy\/,\quad x\in\R\/, f\in L^2(\R)\/.
\end{equation*}
\begin{lem}
\label{compact}
For every $t>0$, the operators $T_t$ are compact.
\end{lem}
\begin{proof}
\cite{KK}, Lemma 1.
\end{proof}

\begin{lem}
Let $u(t,x,y)$ be the integral kernel of the Feynman-Kac semigroup $\{T_t: t \geq 0\}$. The following properties
hold:
\begin{enumerate}
\item
for every $t>0$ the function $u(t,\cdot,\cdot)$ is continuous, strictly positive, and bounded on $\R\times \R$
\item
the semigroup is intrinsically ultracontractive, i.e., there exists $C(t)>0$ such that $u(x,y,t) \leq C(t)\varphi_1(x)
\varphi_1(y)$, for all $t >0$ and $x,y \in \R$, where $\varphi_1$ is the first eigenfunction of
$\sqrt{-\frac{d^2}{dx^2}}+x^2$.
\end{enumerate}
\label{uprop}
\end{lem}
\begin{proof}
\cite{KK}.
\end{proof}


By Lemma \ref{compact} above the spectrum of $H$ is purely discrete and there exists an orthonormal basis
in $L^2(\R)$ consisting of eigenfunctions $\varphi_n$ such that $T_t\varphi_n = e^{-\lambda_n t}\varphi_n$,
 where $0<\lambda_1<\lambda_2\leq \lambda_3\leq \ldots \to \infty$ are the eigenvalues. Using the relation
 between the semigroup $\{T_t: t \geq 0\}$ and its generator $H$ we get in strong sense
\begin{equation*}
H\varphi_n(x) = \lim_{t\downarrow 0} \frac{T_t\varphi_n(x)-\varphi_n(x)}{t} =
\lim_{t\downarrow 0} \frac{e^{-\lambda_n t}-1}{t}\varphi_n(x) = -\lambda_n\varphi_n(x)\/,\quad x\in\R\/,
\end{equation*}
which means that the functions $\seq\varphi$ are also eigenfunctions of the Schr\"odinger operator $H$ and
\begin{equation}
\label{Eigenproblem}
\sqrt{-\frac{d^2}{dx^2}}\varphi_n(x)+x^2\varphi_n(x) = \lambda_n\varphi_n(x)\/,\quad x\in  \R\/.
\end{equation}

We conclude this section by discussing some basic regularity properties of the eigenfunctions of $H$.
\begin{lem}
For every $n=1,2,\ldots$ we have $\varphi_n\in L^1(\R)$.
\label{L1}
\end{lem}
\begin{proof}
   For every $x\in \R$ define
   \begin{equation*}
      f(x) = \ex^x \left[e^{-\int_0^1 X_s^2ds}\right]
   \end{equation*}
   and denote by $\tau:=\inf\{s>0: X_s\leq -1\}$ the first exit time of the process $\pro X$ from the
   half-line $(-1,\infty)$. Then, for every $x\in\R$ such that $x>2$ we have
   \begin{eqnarray*}
     f(x) &=& \ex^x \left[e^{-\int_0^1 X_s^2ds} \right] = \ex^0 \left[e^{-\int_0^1 (X_s+x)^2ds} \right]\\
     &\leq& \ex^0 \left[e^{-\int_0^{1\wedge \tau}(X_s+x)^2ds}\right]\\
     &\leq& \ex^0 \left[e^{-(x-1)^2\int_0^{1\wedge \tau}ds}\right] = \ex^0 \left[e^{-(x-1)^2(1\wedge \tau)}\right]\\
     &\leq& e^{-(x-1)^2}+\int_{1}^\infty e^{-(x-1)^2z}g(z)dz\/,
   \end{eqnarray*}
   where $g(z)$ is a density function of the random variable $\tau$. The explicit formula for $g$ was derived by Darling \cite{D} (compare also \cite{B,KKMS}). There exists a constant $c_1>0$ such that for every $x>2$ we have $g(x)<c_1 x^2$ (\cite{B}, p. 286). Thus
   \begin{eqnarray*}
      f(x) &\leq& e^{-(x-1)^2} + c_1 \int_1^\infty e^{-(x-1)^2z}z^2dz \leq
      e^{-(x-1)^2}\left(1+c_1\int_0^\infty e^{-(x-1)^2u}(u+1)^2du\right)\leq c_2 e^{-(x-1)^2}
   \end{eqnarray*}
   whenever $x>2$. The same argument for $x<-2$ shows that there is a constant $c_3$ such that for every $|x|>2$
   \begin{equation*}
      f(x) \leq c_3 e^{-(|x|-1)^2}\/.
   \end{equation*}

   Finally, we get
   \begin{eqnarray*}
     \int_{\R} |\varphi_n(y)|dy &=& e^{\lambda_n}\int_{\R}|T_1\varphi_n(y)|dy \leq e^{\lambda_n}\int_{\R}\int_\R u(1,x,y)|\varphi_n(x)|dxdy\\
     &=& e^{\lambda_n}\int_{\R}f(x)|\varphi_n(x)|dx = e^{\lambda_n}\(\int_{|x|\leq 2}+\int_{|x|>2}\)f(x)|\varphi_n(x)|dx\\
     &\leq& e^{\lambda_n} \(\int_{-2}^2 |\varphi_n(x)|dx + c_3 \int_{|x|>2}e^{-(|x|-1)^2}|\varphi_n(x)|dx\)< \infty\/,
   \end{eqnarray*}
   where the last inequality is a consequence of the fact that $\varphi_n\in L^2(\R)$.
\end{proof}

In fact, a stronger property is true, which we will need below.
\begin{lem}
\label{L1+}
We have that $x^2\varphi_n\in L^1(\R)$, for all $n = 1,2,...$
\end{lem}
\begin{proof}
For $V(x)=x^2$ the estimate
$$
\varphi_1(x) \leq \frac{C}{V(x)|x|^{d+\alpha}}
$$
holds for the first eigenfunction (ground state) $\varphi_1(x)$, see \cite{KK}, Theorem 1. Since the
Feynman-Kac semigroup is intrinsically ultracontractive by Lemma \ref{uprop} above,
\begin{eqnarray}
\label{Eigenfunction_estimate}
|\varphi_n(x)| \leq \frac{c_n}{x^4}\/,\quad |x|>1
\end{eqnarray}
follows. Then the proof of the claim is straightforward.
\end{proof}

\subsection{Eigenvalues}
We will determine the functions $\varphi_n$ and corresponding eigenvalues $\lambda_n$ starting from the relation
(\ref{Eigenproblem}). Denote the Fourier transform of eigenfunctions by
\begin{equation*}
y_{\lambda_n}(x) := \widehat{\varphi}_n(x) = \frac{1}{\sqrt{2\pi}} \int_\R e^{i x z}\varphi_n(z)dz\/.
\end{equation*}
Note that $y_{\lambda_n}$ is well-defined since $\varphi_n \in L^1(\R)$ by Lemma \ref{L1}. Moreover, Lemma
\ref{L1+} implies that $y_{\lambda_n} \in C^2(\R)$, $n=1,2,...$ By performing Fourier transform in (\ref{Eigenproblem})
\begin{equation}
-y_{\lambda_n}''(x)+|x|y_{\lambda_n}(x) = \lambda_n y_{\lambda_n}(x)
\end{equation}
is obtained. We are looking for $y_\lambda\in{C}^2(\R)\cap L^1(\R)$ satisfying
\begin{equation}
   \label{FTEquation}
   -y_\lambda''(x)+|x|y_\lambda(x) = \lambda y_\lambda(x)\/,\quad \lambda>0\/.
\end{equation}
Notice that if a function $y_\lambda(x)$ is a solution of (\ref{FTEquation}), then the function $y_\lambda(-x)$ is also a solution.
Hence it suffices to consider equation (\ref{FTEquation}) only for $x>0$ and construct even and odd solutions on the whole real line.

For $x>0$ equation (\ref{FTEquation}) takes the form
\begin{equation*}
   y_\lambda''(x)-(x-\lambda)y_\lambda(x) = 0\/.
\end{equation*}
On substituting $z=x-\lambda$, $f_\lambda(z) = y_\lambda(x)$ the equation reduces to the Airy differential equation
(\ref{AiryEq}) below
\begin{eqnarray*}
   f_\lambda''(z)-zf_\lambda(z) = 0\/.
\end{eqnarray*}
(For a discussion of Airy functions see the Appendix.) The general solution of the above equation is thus obtained as
\begin{eqnarray*}
   y_\lambda(x) =c_1 \Ai(x-\lambda)+c_2\Bi(x-\lambda)\/.
\end{eqnarray*}
The facts that $y\in L^1(\R)$ and the function $\Bi(z)$ tends to infinity when $z\to\infty$ (see (\ref{BiDefn})) imply $c_2=0$. Without loss of generality we can assume that $c_1=1$.

Finally, to obtain an even function on the whole real line we put
\begin{equation}
  \label{y_formula_even}
  y_\lambda(x) = \Ai(|x|-\lambda)\/, \quad x\in\R.
\end{equation}
We furthermore require that the right derivative at $0$ is zero
\begin{equation*}
   \lim_{x\to0+}y'_\lambda(x) = 0
\end{equation*}
which is equivalent to
\begin{eqnarray}
    \label{Eigenvalues_even_equation}
   \Ai'(-\lambda) = 0\/.
\end{eqnarray}
For odd functions we have
\begin{equation}
  \label{y_formula_odd}
  y_\lambda(x) = \sgn(x) \Ai(|x|-\lambda)\/,\quad x\in\R
\end{equation}
and we require
\begin{eqnarray*}
   \lim_{x\to 0+}y_\lambda(x) &=& 0
\end{eqnarray*}
or equivalently
\begin{eqnarray}
    \label{Eigenvalues_odd_equation}
    \Ai(-\lambda) = 0\/.
\end{eqnarray}
Conditions (\ref{Eigenvalues_even_equation}) and (\ref{Eigenvalues_odd_equation}) together with
equation (\ref{FTEquation}) imply that the functions $y_\lambda$ defined on the real line by
(\ref{y_formula_even}) and (\ref{y_formula_odd}) belong to $C^2(\R)$.

By the Parseval equality and the fact that $(x\Ai^2(x)-(\Ai(x))^2)'=\Ai(x)$, which is an easy
consequence of the Airy equation (\ref{AiryEq}), we get
\begin{eqnarray*}
   \int_{-\infty}^\infty \varphi_n^2(x)dx &=& \int_{-\infty}^\infty (\widehat{\varphi_n})^2(x)dx
   = 2\int_0^\infty \Ai(x+\lambda_n)^2dx
   = 2\left((\Ai'(-\lambda_n))^2+\lambda_n \Ai^2(-\lambda_n)\right)\/.
\end{eqnarray*}
We have thus proved the
following result.
\begin{thm}
\label{eigenval}
The eigenvalues for the problem (\ref{Eigenproblem}) are given by
\begin{eqnarray*}
\lambda_{2k-1} = -a_k'\quad k=1,2\ldots\/,\\
\lambda_{2k} = -a_k\/,\quad k=1,2\ldots\/,
\end{eqnarray*}
where $a_k$ and $a_k'$ denote the zeroes of the functions $\Ai$ and $\Ai'$ in decreasing order. They are
all simple, the eigenfunctions $\varphi_{2k-1}(x)$ are even and $\varphi_{2k}(x)$ are odd. Furthermore,
the Fourier transforms of the $L^2$-normalized eigenfunctions are given by
\begin{equation*}
\widehat{\varphi_n}(x) = \left\{
\begin{array}{cl}
\dfrac{\Ai(|x|-\lambda_n)}{\sqrt{2\lambda_n}\Ai(-\lambda_n)}& n=1,3,5,\ldots\\ \\
\dfrac{\sgn(x) \Ai(|x|-\lambda_n)}{\sqrt{2}\Ai'(-\lambda_n)}& n=2,4,6,\ldots
\end{array}
\right.\/,\quad x\in\R.
\end{equation*}
\end{thm}
Making use of the asymptotic expansions and estimates for the zeroes of the Airy function and its
derivative \cite{PS,H} yields
\begin{cor}
\label{Eigenvalues:asymp}
We have
\begin{eqnarray*}
   \lambda_{2k-1} &\sim&g\(\frac{3}{8}\pi(4k-3)\),\quad k\to \infty\/,\\
   \lambda_{2k} &\sim& f\(\frac{3}{8}\pi (4k-1)\),\quad k\to \infty\/.
\end{eqnarray*}
where
\begin{eqnarray}
   \label{eigenvalues:asymptotics:odd}
   g(t) &=& t^{2/3}\(1-\frac{7}{48}t^{-2}+\frac{35}{288}t^{-4}-\frac{181223}{207360}t^{-6}+\frac{18683371}{1244160}t^{-8}-\frac{91145884361}{191102976}t^{-10}\)\/,\\
   \label{eigenvalues:asymptotics:even}
   f(t) &=& t^{2/3}\(1+\frac{5}{48}t^{-2}-\frac{5}{36}t^{-4}+\frac{77125}{82944}t^{-6}- \frac{108056875}{6967296}t^{-8}+\frac{162375596875}{334430208}t^{-10}\)\/.
\end{eqnarray}
Moreover, we have
\begin{eqnarray*}
   \lambda_{2k-1}&\leq& \left(\frac{3\pi}{8}(4k-1)\right)^{2/3}\/,\quad k=1,2,\ldots\/,\\
   \left(\frac{3\pi}{8}(4k-1)\right)^{2/3}\leq \lambda_{2k} &\leq& \left(\frac{3\pi}{8}(4k-1)\right)^{2/3}\left(1+\frac{3}{2} \arctan\left(\frac{5}{18\pi(4k-1)}\right)\right)\/,\quad k=1,2,\ldots
\end{eqnarray*}
\end{cor}

\begin{rem}
A numerical calculation for the first few eigenvalues gives
\begin{eqnarray*}
   \lambda_1 &\cong& 1.01879297164747\\
   \lambda_2 &\cong& 2.33810741045976\\
   \lambda_3 &\cong& 3.24819758217983\\
   \lambda_4 &\cong& 4.08794944413097\\
   \lambda_5 &\cong& 4.82009921117874\\
   \lambda_6 &\cong& 5.52055982809555\/.
\end{eqnarray*}
\end{rem}

Using the above asymptotic formulae we can derive a result on the asymptotic behaviour of the
trace of the semigroup at zero.
\begin{thm}
\label{trestim}
We have
\begin{eqnarray*}
   \lim_{t\to0^{+}}t^{3/2}\sum_{n=1}^\infty e^{-\lambda_n t} = \frac{1}{\sqrt{\pi}}\/.
\end{eqnarray*}
\end{thm}
\begin{proof}
   We divide the series into two components
   \begin{equation*}
      F(t) = \sum_{k=1}^\infty e^{-\lambda_{2k-1}t} = \sum_{k=1}^\infty e^{a_k't}\/,\quad G(t) = \sum_{k=1}^\infty e^{-\lambda_{2k}t} = \sum_{k=1}^\infty e^{a_kt}\/.
   \end{equation*}
   By (\ref{eigenvalues:asymptotics:odd}) we get that for every $t\in [0,1]$
   \begin{eqnarray*}
      \sum_{k=1}^\infty e^{-(3/2\pi(k-3/4)^{2/3}t)}\leq F(t)\leq  e^{-ct}\sum_{k=1}^\infty e^{-(3/2\pi(k-3/4)^{2/3}t)}
   \end{eqnarray*}
   for some constant $c>0$. Moreover, the function $x\mapsto e^{-(3/2\pi(x-3/4)^{2/3}t)}$ is strictly positive and non-increasing on $[1,\infty)$. Thus
   \begin{eqnarray*}
       e^{-ct}\int_1^\infty e^{-(3/2\pi(x-3/4)^{2/3}t)} dx\leq F(t) \leq \left(e^{-(3/8\pi)^{2/3}t} + \int_1^\infty e^{-(3/2\pi(x-3/4)^{2/3}t)} dx\right)\/.
   \end{eqnarray*}
   A substitution yields
   \begin{equation*}
      \int_1^\infty e^{-(3/2\pi(x-3/4)^{2/3}t)} dx = \frac{1}{t^{3/2}\pi} \int_{(3/8\pi)^{2/3}t}^\infty e^{-u}u^{1/2}du
   \end{equation*}
   and we obtain $\lim_{t\to 0^{+}} t^{3/2}F(t) = \pi^{-1}\Gamma(3/2)=\frac{1}{2\sqrt{\pi}}$. In the same way we get $\lim_{t\to 0^{+}} t^{3/2}G(t) = \frac{1}{2\sqrt{\pi}}$ and this completes the proof.
\end{proof}

Using the estimates for the eigenvalues given in Corollary \ref{Eigenvalues:asymp} we obtain an estimate for the spectral gap.
\begin{cor}
\label{specgap}
We have
\begin{eqnarray*}
   \lambda_2-\lambda_1\geq \left(\frac{3\pi}{8}\right)^{2/3}(3^{2/3}-1)\/.
\end{eqnarray*}
\end{cor}

\subsection{Eigenfunctions}
Next we derive the asymptotic behaviour of the eigenfunctions. We use the notation $p_n, q_n$ as in (\ref{AiDerivatives})
below.
\begin{thm}\label{Eigenfunctions:asymptotics}
For every $k=1,\ldots$ and $N=2,3,\ldots$ we have
\begin{eqnarray*}
\varphi_{2k-1}(z) = \sqrt{\frac{2}{-a_k'}}\(\frac{p_3(a_k')}{z^4}-\frac{p_5(a_k')}{z^6}+\ldots+(-1)^{N}\frac{p_{2N-1}(a_k')}{z^{2N}}\)+O\(\frac{1}{z^{2N+2}}\)
\/,\quad \textrm{as }|z|\to\infty\/.
\end{eqnarray*}
For every $k=1,2,\ldots$ and $N=2,3,\ldots$ we have
\begin{eqnarray*}
\varphi_{2k}(z) = \sqrt{2}\(\frac{q_4(a_k)}{z^{5}}-\frac{q_6(a_k)}{z^7}+\ldots+(-1)^N\frac{q_{2N}(a_k)}{z^{2N+1}}\)+O\(\frac{1}{z^{2N+3}}\)\/,
\quad \textrm{as }|z|\to\infty\/.
\end{eqnarray*}
\end{thm}
\begin{proof}
  For $k=1,2,\ldots$, we have $\lambda_{2k-1}=-a_k'$,  and
 \begin{eqnarray}
    \label{eigenfunction:odd:integral}
    \varphi_{2k-1}(z) = \sqrt{\frac{2}{-a_k'}}\frac{1}{\Ai(a_k')}\int_0^\infty \Ai(u+a_k')\cos zu\, du\/.
 \end{eqnarray}
 Integration by parts ($2N+2$ times) together with (\ref{Airy:asymp:infty}), (\ref{AiryPrime:asymp:infty}) and (\ref{AiDerivatives}) give
 \begin{eqnarray*}
     \int_0^\infty \Ai(u+a_k')\cos zu\, du &=& \sum_{s=0}^{N} \left.\frac{\sin zu}{z^{2s+1}} (-1)^s \Ai^{(2s)}(u+a_k')\right|_{0}^\infty+\\
     && + \sum_{s=1}^{N} \left.\frac{\cos zu}{z^{2s}}(-1)^{s-1} \Ai^{(2s-1)}(u+a_k')\right|_{0}^\infty +R_N(z)\\
    &=& \sum_{s=1}^{N} \frac{1}{z^{2s}}(-1)^s \Ai^{(2s-1)}(a_k')+R_N(z)\\
    &=& \sum_{s=1}^{N} \frac{p_{2s-1}(a_k')}{z^{2s}}(-1)^s \Ai(a_k')+R_N(z)\/,
 \end{eqnarray*}
 where
 \begin{eqnarray*}
    R_N(z) &=& \left.\frac{\cos zu}{z^{2N+2}}\Ai^{(2N+1)}(u+a_k')\right|_0^\infty+\frac{1}{z^{2N+2}}\int_0^\infty \Ai^{(2N+2)}(u+a_k')\cos zu du\\
    &=& -\frac{p_{2N+1}(a_k')\Ai(a_k')}{z^{2N+2}}+\frac{1}{z^{2N+2}}\int_0^\infty \Ai^{(2N+2)}(u+a_k')\cos zu du\/.
 \end{eqnarray*}
 Using the asymptotic relations (\ref{Airy:asymp:infty}) and (\ref{AiryPrime:asymp:infty}) together with formula (\ref{AiDerivatives}) we get
 \begin{eqnarray*}
    |R_N(z)| \leq \frac{1}{|z|^{2N+2}}\(|p_{2N+1}(a_k')\Ai(a_k')|+ \int_0^\infty |\Ai^{(2N+2)}(u+a_k')| du\)= \frac{c_{2k-1,N}}{|z|^{2N+2}}\/.
 \end{eqnarray*}
 Notice that $p_1(x)\equiv 0$, which completes the proof for this case.

 For $k=1,2,\ldots$ we have
 \begin{equation*}
    \varphi_{2k}(z) = \frac{\sqrt{2}}{\Ai'(a_k)}\int_0^\infty \Ai(u+a_k)\sin zu du.
 \end{equation*}
 Similar arguments give
 \begin{eqnarray*}
    \frac{\Ai'(a_k)}{\sqrt{2}} \varphi_{2k}(z) &=& \sum_{s=0}^{N} \left.\frac{\cos zu}{z^{2s+1}}(-1)^{s+1} \Ai^{(2s)}(u+a_k)\right|_{0}^\infty
    + \sum_{s=0}^{N} \left.\frac{\sin zu}{z^{2s+2}} (-1)^s \Ai^{(2s+1)}(u+a_k)\right|_{0}^\infty +R_N(z)\\
    &=& \sum_{s=0}^{N} \frac{(-1)^s}{z^{2s+1}} \Ai^{(2s)}(a_k)+R_N(z)\\
    &=& \sum_{s=1}^{N} \frac{(-1)^s q_{2s}(a_k)}{z^{2s+1}} \Ai'(a_k)+R_N(z)\/,
 \end{eqnarray*}
 where
 \begin{eqnarray*}
    R_N(z) &=& \left.\frac{\cos zu}{z^{2N+3}}(-1)^{N+2}\Ai^{(2N+2)}(u+a_k)\right|_0^\infty  + \frac{(-1)^{N+1}}{z^{2N+3}}\int_0^\infty \Ai^{(2N+3)}(u+a_k)\cos zu du\\
    &=& \frac{1}{z^{2N+3}}\(q_{2N+3}(a_k) \Ai'(a_k)+\int_0^\infty \Ai^{(2N+3)}(u+a_k)\cos zu du\)\/.
 \end{eqnarray*}
Thus we get $|R_N(z)|\leq c_{2k,N}|z|^{-2N-3}$. Finally, notice that $q_2(x)\equiv 0$. This completes the proof.
\end{proof}
\begin{thm}
   The eigenfunctions $\varphi_n$ are analytic functions on $\R$. Their Maclaurin expansions are given by
   \begin{eqnarray*}
       \varphi_{2k-1}(x) &=& \sqrt{\frac{2}{-a_k'}}\frac{1}{\Ai(a_k')}\,\sum_{m=0}^\infty \frac{w_{2m}(a_k')(-1)^m}{(2m)!}\, x^{2m}\/,\\
       \varphi_{2k}(x) &=&  \frac{\sqrt{2}}{\Ai'(a_k)}\,\sum_{m=0}^\infty \frac{w_{2m+1}(a_k)(-1)^m}{(2m+1)!}x^{2m+1}\/,
   \end{eqnarray*}
   where $k=1,2,\ldots$ and
   \begin{eqnarray*}
       w_n(x) = \int_0^\infty \Ai(u+x)u^ndu\/.
   \end{eqnarray*}
\end{thm}
\begin{proof}
 For $k=1,2,\ldots$ we have
 \begin{eqnarray*}
   \varphi_{2k-1}(x) &=& \sqrt{\frac{2}{-a_k'}}\frac{1}{\Ai(a_k')} \int_0^\infty \Ai(u+a_{k}')\cos xu du\\
      &=& \sqrt{\frac{2}{-a_k'}}\frac{1}{\Ai(a_k')}\int_0^\infty \Ai(u+a_{k}')\sum_{m=0}^\infty \frac{(xu)^{2m}}{(2m)!}(-1)^m du\/.
 \end{eqnarray*}
  Moreover, by (\ref{Airy:asymp:infty}) it is seen that there exists a constant $c_k>0$ such that $|\Ai(u+a_{k}')|< c_k e^{-\frac{2}{3}u^{3/2}}$ for all $u>0$. Hence
   \begin{eqnarray*}
      \sum_{m=0}^\infty \int_0^\infty \left|\frac{x^{2m}u^{2m}}{(2m)!}(-1)^m \Ai(u+a_{k}') \right| du &=& \sum_{m=0}^\infty \frac{|x|^{2m}}{(2m)!}\int_0^\infty |\Ai(u+a_k')| u^{2m}du\\
      &\leq& c_k \sum_{m=0}^\infty \frac{|x|^{2m}}{(2m)!} \int_{0}^\infty e^{-\frac{2}{3}u^{3/2}} u^{2m}du\\
      &=& c_k \sum_{m=0}^\infty \frac{(|x|^{2})^m}{(2m)!} \(\frac{3}{2}\)^{\frac{4m-1}{3}}\Gamma\left(\frac{4m+2}{3}\right)\/.
   \end{eqnarray*}
   By putting $d_m = \frac{1}{(2m)!} \(\frac{3}{2}\)^{\frac{4m-1}{3}}\Gamma\left(\frac{4m+2}{3}\right)$ and making use of Stirling's formula $\Gamma(x) \sim \sqrt{2\pi} e^{-x} x^{x-1/2}$, as $x\to \infty$ we obtain
   \begin{eqnarray*}
      \frac{d_m}{d_{m+1}} &=& (2m+1)(2m+2)\left(\frac{2}{3}\right)^{4/3}\frac{\Gamma\left(\frac{4m+2}{3}\right)}{\Gamma\left(\frac{4m+6}{3}\right)}\\
      &\sim&  (2m+1)(2m+2)\left(\frac{2}{3}\right)^{4/3}e^{4/3}\left(\frac{4m+2}{4m+6}\right)^{\frac{4m+2}{3}-\frac{1}{2}} \left(\frac{3}{4m+6}\right)^{4/3}\stackrel{m\to\infty}{\longrightarrow} \infty\/.
   \end{eqnarray*}
   Thus Fubini's theorem applies to get
   \begin{eqnarray*}
      \varphi_{2k-1}(x) &=& \sqrt{2}\sum_{m=0}^\infty \frac{(-1)^m}{\sqrt{{-a_k'}}(2m)!}\left(\frac{1}{\Ai(a_k')}\int_0^\infty \Ai(u+a_k')u^{2m}du\right)x^{2m}\/.
   \end{eqnarray*}
   Similar arguments give
      \begin{eqnarray*}
      \varphi_{2k}(x) &=& \sqrt{2}\sum_{m=0}^\infty \frac{(-1)^m}{(2m+1)!}\left(\frac{1}{\Ai'(a_k)}\int_0^\infty \Ai(u+a_k)u^{2m+1}du\right) x^{2m+1}\/.
   \end{eqnarray*}
\end{proof}
\begin{cor}
\label{zeroes}
Every eigenfunction $\varphi_n$ has a finite number of zeroes.
\end{cor}
\begin{proof}
Using the asymptotic expansions in Theorem \ref{Eigenfunctions:asymptotics} it is easily seen that for every $n$ there exists $A_n>0$ such that $\sup_{|x|>A_n}|\varphi_n(x)|>0$. This means that all zeroes of $\varphi_n$ are in $[-A_n,A_n]$. Since the function $\varphi_n$ is analytic, the set of its zeroes is finite.
\end{proof}

\begin{thm}
\label{UB:thm}
   The eigenfunctions $\varphi_n$ are uniformly bounded.
\end{thm}
\begin{proof}
   We begin with the case of odd eigenfunctions $\varphi_n$, where $n=2k-1$, $k=1,2,\ldots$; the proof for the even eigenfunctions is similar. It suffices to consider $x\geq0$. We have
   \begin{eqnarray*}
      \varphi_n(x) &=& \sqrt{\frac{2}{-a_k'}}\frac{1}{\Ai(a_k')}\int_0^\infty \Ai(u+a_k')\cos zu\, du\/.
   \end{eqnarray*}
   Using the asymptotic formulae (\ref{eigenvalues:asymptotics:odd}), (\ref{Airy:asymp:akprime}) for $-a_k'$ and $\Ai(a_k')$ respectively we get that
   \begin{equation}
       \label{normalization:odd:asymp}
       \frac{1}{\sqrt{-a_k'}\Ai(a_k')} = O(k^{-1/6})
   \end{equation}
   We have
   \begin{eqnarray*}
      \int_0^\infty \Ai(u+a_k')\cos xu du = \int_{-a_1'}^{-a_k'}\Ai(u)\cos x(u+a_k')du +\int_{a_1'}^\infty \Ai(u)\cos x(u-a_k')du\/.
   \end{eqnarray*}
   The Airy function is non-negative on $[a_1',\infty]$ and thus the absolute value of the second integral is uniformly bounded by $\int_{a_1'}^\infty \Ai(u)du<\infty$. The first integral is a sum of two integrals $I_1(x)$ and $I_2(x)$ where
   \begin{eqnarray*}
      I_1(x) &=& \int_{-a_1'}^{-a_k'}\(\Ai(u)-\frac{\sin \(\frac{2}{3}u^{3/2}+\frac{\pi}{4}\)}{u^{1/4}\sqrt{\pi}}\)\cos x(u+a_k')du\\
      I_2(x) &=& \frac{1}{\sqrt{\pi}}\int_{-a_1'}^{-a_k'}\sin \(\frac{2}{3}u^{3/2}+\frac{\pi}{4}\)\frac{\cos x(u+a_k')}{u^{1/4}}du\/.
   \end{eqnarray*}
   Using the asymptotic expansion for the Airy function (\cite{AS} 10.4.60 p.448) we get
   \begin{eqnarray*}
      \Ai(u)-\frac{\sin \(\frac{2}{3}u^{3/2}+\frac{\pi}{4}\)}{u^{1/4}\sqrt{\pi}} = O\left(\frac{1}{u^{7/4}}\right)\/,
   \end{eqnarray*}
   as $u\to \infty$. Thus there is a constant $c_1>0$ such that
   \begin{eqnarray*}
      |I_1(x)|\leq c_1\int_{-a_1'}^\infty u^{-7/4}du\/.
   \end{eqnarray*}
   The integral $2\sqrt{\pi} I_2(x)$ can be rewritten as the sum $I_3(x)+I_4(x)$, where
   \begin{eqnarray*}
       I_3(x) &=& \int_{-a_1'}^{-a_k'} (u^{1/2}+x)\sin\(\frac{2}{3}u^{3/2}+\frac{\pi}{4}+xu+xa_k'\)\frac{du}{(u^{1/2}+x)u^{1/4}}\\
       I_4(x) &=& \int_{-a_1'}^{-a_k'} \sin\(\frac{2}{3}u^{3/2}+\frac{\pi}{4}-xu-xa_k'\)\frac{du}{u^{1/4}}\/.
   \end{eqnarray*}
    The term $I_3$ is uniformly bounded by Lemma \ref{UB:lemma} with
    $$
    f(x,u) = (u^{1/2}+x)\sin\(\frac{2}{3}u^{3/2}+\frac{\pi}{4}+xu+xa_k'\)
    $$
    and
    $$
    g(x,u) = (u^{1/2}+x)^{-1}u^{-1/4}.
    $$
    To deal with the term $I_4$ we need to consider several cases. Let $(x-1)^2>-a_k'$ or $(x+1)^2<-a_1'$ and rewrite $I_4$ in the form
    \begin{equation*}
       I_4(x) = \int_{-a_1'}^{-a_k'} (u^{1/2}-x)\sin\(\frac{2}{3}u^{3/2}+\frac{\pi}{4}-xu-xa_k'\)\frac{du}{u^{1/4}(u^{1/2}-x)}.
    \end{equation*}
An application of Lemma \ref{UB:lemma} with
$$
f(x,u) = (u^{1/2}-x)\sin\(\frac{2}{3}u^{3/2}+\frac{\pi}{4}-xu-xa_k'\)
$$
and
$$
g(x,u) = (u^{1/2}-x)^{-1}u^{-1/4}
$$
implies that $I_4$ is uniformly bounded for $x>\sqrt{-a_k'}+1$ and $k=2,3,\ldots$. For the case $-a_1'< (x-1)^2\leq-a_k'\leq(x+1)^2$ we have
\begin{eqnarray*}
  I_4(x) &=& \int_{-a_1'}^{(x-1)^2} (u^{1/2}-x)\sin\(\frac{2}{3}u^{3/2}+\frac{\pi}{4}-xu-xa_k'\)\frac{du}{u^{1/4}(u^{1/2}-x)}  \\
  &&+ \int_{(x-1)^2}^{-a_k'} \sin\(\frac{2}{3}u^{3/2}+\frac{\pi}{4}-xu-xa_k'\)\frac{du}{u^{1/4}}\/.
\end{eqnarray*}
  The first integral above is uniformly bounded by the same argument as in the previous case. The absolute value of the second integral is bounded by
  \begin{equation*}
     \int_{(x-1)^2}^{-a_k'}\frac{du}{u^{1/4}}\leq \int_{(x-1)^2}^{(x+1)^2}\frac{du}{u^{1/4}} = \frac{3}{4}((x+1)^{3/4}-(x-1)^{3/4})\leq c_1 x^{1/4}\leq c_2 k^{1/6}\/.
  \end{equation*}
  The last inequality follows from the fact that $(x-1)^{2}\leq -a_k'$ and the asymptotic expansion for $a_k'$. In the case $-a_1'< (x-1)^2\leq(x+1)^2<-a_k'$ we split up $I_4$ as
  \begin{eqnarray*}
     I_4(x) &=& \left(\int_{-a_1'}^{(x-1)^2}+\int_{(x+1)^2}^{-a_k'}\right) (u^{1/2}-x)\sin\(\frac{2}{3}u^{3/2}+\frac{\pi}{4}-xu-xa_k'\)\frac{du}{u^{1/4}(u^{1/2}-x)}  \\
  &&+ \int_{(x-1)^2}^{-a_k'} \sin\(\frac{2}{3}u^{3/2}+\frac{\pi}{4}-xu-xa_k'\)\frac{du}{u^{1/4}}\/.
  \end{eqnarray*}
    By Lemma \ref{UB:lemma} the first two integrals are uniformly bounded following by the same argument as before. The last integral is bounded by $c_2k^{1/6}$. Finally, when $(x-1)^2<-a_1'<(x+1)^2<-a_k'$ we have
  \begin{eqnarray*}
     I_4(x) &=& \int_{-a_1'}^{(x+1)^2} \sin\(\frac{2}{3}u^{3/2}+\frac{\pi}{4}-xu-xa_k'\)\frac{du}{u^{1/4}}  \\
     &&+ \int_{(x-1)^2}^{-a_k'} (u^{1/2}-x)\sin\(\frac{2}{3}u^{3/2}+\frac{\pi}{4}-xu-xa_k'\)\frac{du}{u^{1/4}(u^{1/2}-x)}\/.
  \end{eqnarray*}
  The absolute value of the first integral is estimated as
  \begin{eqnarray*}
     \int_{-a_1'}^{(x+1)^2}\frac{du}{u^{1/4}}\leq \int_{-a_1'}^{(\sqrt{-a_1'}+2)^2}\frac{du}{u^{1/4}}.
  \end{eqnarray*}
 The last term can be uniformly bounded by one more application of Lemma \ref{UB:lemma}. Hence we obtain
  \begin{eqnarray*}
    \int_0^\infty \Ai(u+a_k')\cos xu du  = O(k^{1/6})
  \end{eqnarray*}
  uniformly in $x\geq 0$. Together with (\ref{normalization:odd:asymp}) this implies that the functions $\varphi_n$ are uniformly bounded.
\end{proof}

\begin{thm}
\label{shapegs}
The ground state $\varphi_1$ is decreasing on $(0,\infty)$. Moreover, there exist $x_1>x_0>0$ such that $\varphi_1$ is concave on $[-x_0,x_0]$ and is convex on $(-\infty,-x_1]$ and $[x_1,\infty)$.
\end{thm}
\begin{proof}
   For $0<x<y$, let $\pro {R^{(1)}}$, $\pro {R^{(2)}}$ be two squared-Bessel processes of dimension 1 and
   with index $\nu=-1/2$, such that $R^{(1)}_0=x^2$ and $R^{(1)}_0=y^2$, and let $\pro\eta$ be a $1/2$-stable
   subordinator independent from $R^{(1)}$ and $R^{(2)}$. Using the comparison theorem (see \cite{RY}, Chapter
   IX, Theorem 3.7) we get $R^{(1)}_t\leq R^{(2)}_t$ for all $t\geq 0$ with probability $1$. Thus
   $R^{(1)}_{\eta_t}\leq R^{(2)}_{\eta_t}$, for all $t\geq 0$. However, the process $R^{(1)}_{\eta_t}$ is
   $X_t^2$ starting form $x^2$, and $R^{(1)}_{\eta_t}$ is $X_t^2$ starting form $y^2$. Hence
   \begin{eqnarray*}
       \E^x\left[e^{-\int_0^t X_s^2ds}\right] \geq \E^y\left[e^{-\int_0^t X_s^2ds}\right]\/,\quad 0<x<y\/.
   \end{eqnarray*}
   For every $x\in \R$ we have
   \begin{eqnarray*}
     \E^x\left[e^{-\int_0^t X_s^2ds}\right]
     &=& \int_\R u(t,x,y)dy\\
     &=& \int_\R \sum_{n=1}^\infty e^{-\lambda_n t}\varphi_n(x)\varphi_n(y)dy\\
     &=& \sum_{n=1}^\infty e^{-\lambda_n t} \varphi_n(x)\int_\R \varphi_n(y)dy \/.
   \end{eqnarray*}
   This implies
   \begin{eqnarray*}
      \varphi_1(x) =  \left(\int_\R \varphi_n(y)dy\right)^{-1} \lim_{t\to \infty} e^{\lambda_n t}\,\E^x\left[e^{-\int_0^t X_s^2ds}\right]
   \end{eqnarray*}
   and therefore the ground state is non-increasing on $(0,\infty)$ as a limit of non-increasing functions.
   In fact, $\varphi_1$ is strictly decreasing on $(0,\infty)$. This easily follows from the monotonicity of
   $\varphi_1$ proven above, and the fact that $\varphi_1$ is analytic on $\R$.

   Concavity follows from the expression
   \begin{equation*}
      \varphi_1''(0) = -\sqrt{\frac{2}{-a_1'}}\frac{1}{\Ai(a_1')}\int_0^\infty \Ai(u+a_1')u^2du\/.
   \end{equation*}
   We have $\Ai(u+a_1')u^2>0$ on $(0,\infty)$ thus $\varphi_1''(0)<0$. By continuity of $\varphi_1''$ at $0$ it follows that there exists $x_0>0$ such that $\varphi_1''(x)<0$ for all $x\in [-x_0,x_0]$. Moreover, we have
   \begin{eqnarray*}
      \varphi_1(x) &=&      -\sqrt{\frac{2}{-a_1'}}\frac{1}{\Ai(a_1')}\int_0^\infty \Ai(u+a_1')u^2\cos xu du \/.
   \end{eqnarray*}
   Putting $f(u) = \Ai(u+a_1')u^2$ and integrating by parts we obtain
   \begin{eqnarray*}
     \frac{1}{\Ai(a_1')}\int_0^\infty \Ai(u+a_1')u^2\cos xu du = -\frac{f^{(5)}(0)}{x^6}+\frac{1}{x^7}\int_0^\infty f^{(7)}(u)\sin xu du\/,
   \end{eqnarray*}
   where $f^{(5)}(0) = 20\Ai(a_1')$ and $f^{(7)}(u) = P(u)\Ai(u+a_1')+Q(u)\Ai'(u+a_1')$, where $P(u)$ and $Q(u)$ are polynomials. Applying (\ref{Airy:asymp:infty}) and (\ref{AiryPrime:asymp:infty}) we get
   \begin{equation*}
      \lim_{|x|\to\infty}x^6\varphi''_1(x)=20\sqrt{\frac{2}{-a_1}}>0
   \end{equation*}
    and this implies that there exists $x_1>x_0$ such that $\varphi''_1$ is positive on $(-\infty,-x_1]$ and $[x_1,\infty)$.
   This completes the proof.
\end{proof}

\begin{thm}
   The integral kernel $u(t,x,y)$ is jointly continuous on $(0,\infty)\times\R\times \R$. Moreover, there exists a constant $c>1$ such that
   \begin{eqnarray}
   \label{heatkernel:estimates:1}
      \frac{1}{c} \frac{e^{a_1't}}{(1+x^4)(1+y^4)}\leq u(t,x,y)\leq c\frac{e^{a_1't}}{(1+x^4)(1+y^4)}\/,
   \end{eqnarray}
   for every $t>1$ and $x,y\in \R$.
\end{thm}
\begin{proof}
  For every $t_0>0$, using the uniform boundedness of the eigenfunctions $\varphi_n$ given in Theorem \ref{UB:thm}, we have
  \begin{eqnarray*}
     |u(t,x,y)|  \leq \sum_{n=1}^\infty e^{-\lambda_n t}|\varphi_n(x)||\varphi_n(y)|\leq  M \sum_{n=1}^\infty e^{-\lambda_n t}\leq M \sum_{n=1}^\infty e^{-\lambda_n t_0}<\infty\/,\quad x,y\in\R\/,\, t>t_0\/,
  \end{eqnarray*}
  for some constant $M>0$. The boundedness of the last series easily follows from the asymptotic expansions for $\lambda_n$ given in Corollary \ref{Eigenvalues:asymp}. Then the continuity of $u(t,x,y)$ follows from the continuity of $\varphi_n$.

  For $k=1,2,\ldots$ we have
   \begin{eqnarray*}
      \frac{\sqrt{-a_k'}}{\sqrt{2}}\varphi_{2k-1}(x)
      &=&
      \frac{1}{x^4}\left(1+\frac{1}{\Ai(a_k')}\int_0^\infty \Ai^{(4)}(u+a_k')\cos xu du\right) \/,
   \end{eqnarray*}
   where $\Ai^{(4)}(t)= t^2 \Ai(t)+2\Ai'(t)$. From the asymptotic expansions (\cite{AS} 10.4.60 and 10.4.62) one can easily get that
   \begin{equation*}
     |\Ai^{(4)}(t)|\leq c_1 t^2\/,
   \end{equation*}
   with some constant $c_1$. Thus, using (\ref{Airy:asymp:infty}) and (\ref{AiryPrime:asymp:infty}), we get
   \begin{eqnarray*}
      \left|\int_0^\infty \Ai^{(4)}(u+a_k')\cos xu du\right| &\leq& \int_{a_k'}^0|\Ai^{(4)}(t)|dt+\int_0^\infty \Ai^{(4)}(t)dt
      \leq c_1  \int_{a_k'}^0t^2dt+c_2
      \leq c_3 (-a_k')^3
   \end{eqnarray*}
   Similarly, we have
   \begin{eqnarray*}
      \frac{\Ai'(a_k)}{\sqrt{2}}\varphi_{2k}(x) = \frac{1}{x^4}\int_0^\infty \Ai^{(4)}(u+a_k)\sin xu du\/,
   \end{eqnarray*}
   where
   \begin{eqnarray*}
     \left|\int_0^\infty \Ai^{(4)}(u+a_k)\sin xu du\right| &\leq&c_3(-a_k)^3
   \end{eqnarray*}
   Using the asymptotic formulae (\ref{Airy:asymp:akprime}) and (\ref{Airy:asymp:ak})
   we obtain that $|x^4\varphi_{2k-1}(x)|\leq c_4 (k-3/4)^{11/6}$ and  $|x^4\varphi_{2k}(x)| \leq c_4 (k-1/2)^{11/6}$.
   This yields, combined with the asymptotic expansion for $\varphi_1$, the estimates
      \begin{eqnarray*}
      \left|\frac{\varphi_{2k-1}(x)}{\varphi_1(x)}\right| \leq c_5 (k-3/4)^{11/6}\/\quad \mbox{and} \quad
      \left|\frac{\varphi_{2k}(x)}{\varphi_1(x)}\right| \leq c_5 (k-1/2)^{11/6}\/.
   \end{eqnarray*}
   for the ratio of the eigenfunctions, for every $x\in\R$. We have
   \begin{eqnarray*}
    |u(t,x,y)-e^{-\lambda_1 t}\varphi_1(x)\varphi_1(y)|\leq e^{-\lambda_1 t}\varphi_1(x)\varphi_1(y)e^{-(\lambda_2-\lambda_1)t}\sum_{n=2}^\infty e^{-(\lambda_n-\lambda_2)t}\left|\frac{\varphi_n(x)}{\varphi_1(x)}\right|\left|\frac{\varphi_n(y)}{\varphi_1(y)}\right|\/.
   \end{eqnarray*}
   For every $t>0$ the series on the right hand side is uniformly bounded by
   \begin{eqnarray*}
     \sum_{n=2}^\infty e^{-(\lambda_n-\lambda_2)t}\left|\frac{\varphi_n(x)}{\varphi_1(x)}\right|\left|\frac{\varphi_n(y)}{\varphi_1(y)}\right|
      &\leq& c_5\left(\sum_{k=2}^\infty \exp(-c_6(k-3/4)^{2/3})\left(k-\frac{3}{4}\right)^{\frac{11}{3}}\right.\\
      &&\left. \qquad +\sum_{k=1}^\infty \exp(-c_7(k-1/4)^{2/3})\left(k-\frac{1}{4}\right)^{\frac{11}{3}}\right)\/.
   \end{eqnarray*}
   Using the integral test for convergence it is easy to see that the above series converge. Thus the expression
   \begin{equation*}
     e^{-(\lambda_2-\lambda_1)t}\sum_{n=2}^\infty e^{-(\lambda_n-\lambda_2)t}\left|\frac{\varphi_n(x)}{\varphi_1(x)}\right|\left|\frac{\varphi_n(y)}{\varphi_1(y)}\right|
   \end{equation*}
   tends to zero uniformly in $x, y\in\R$ as $t\to\infty$. This proves the estimates (\ref{heatkernel:estimates:1}) for $t>t_0$ and $x\in\R$, $y\in\R$ for some $t_0>1$.

   Now let $t\in[1,t_0]$. Then we have
   \begin{eqnarray*}
      |u(t,x,y)-e^{-\lambda_1 t}\varphi_1(x)\varphi_1(y)|\leq c_8 e^{-\lambda_1 t}\varphi_1(x)\varphi_1(y)
   \end{eqnarray*}
    with a constant $c_8>0$. This provides an upper bound for $u(t,x,y)$, for all $x,y\in\R$. Using the above estimates
    for the ratio $\varphi_n(x)/\varphi_1(x)$, dominated convergence and the asymptotic expressions for the eigenfunctions
    in Theorem \ref{Eigenfunctions:asymptotics} we obtain that for every $t\in[1,t_0]$
   \begin{eqnarray*}
       \lim_{|x|,|y|\to\infty} \frac{u(t,x,y)}{\varphi_1(x)\varphi_1(y)} = -a_1'\sum_{k=1}^\infty \frac{\exp(a_k' t)}{-a_k'}\geq -a_1'\sum_{k=1}^\infty \frac{\exp(a_k' t_0)}{-a_k'}>0\/.
   \end{eqnarray*}
   The function $e^{-\lambda_1 t}$ is comparable with a constant on $[1,t_0]$, whence
   \begin{eqnarray*}
      u(t,x,y)\geq c e^{-\lambda_1 t}\varphi_1(x)\varphi_1(y)
   \end{eqnarray*}
   for $t\in[1,t_0]$, $|x|>x_0$ and $|y|>y_0$ for suitable $x_0,y_0>0$.

   Notice that for every $a\in[-y_0,y_0]$ and $s\in[1,t_0]$
   \begin{eqnarray*}
     \lim_{|x|,\to \infty, y\to a, t\to s} \frac{u(t,x,y)}{\varphi_1(x)\varphi_1(y)} = \sqrt{-a_1'}\sum_{k=1}^\infty \frac{e^{a_k's}}{\sqrt{-a_k'}}\frac{\varphi_{2k-1}(a)}{\varphi_1(a)}>0\/.
   \end{eqnarray*}
   Positivity of the limit is a consequence of the well-known general estimate $u(t,x,y)\geq c_t
   \varphi_1(x)\varphi_1(y)$ derived from intrinsic ultracontractivity, with a constant $c_t>0$.
   Thus there exist $\varepsilon_s,\varepsilon_a>0$, $x_{s,a}>0$ and a constant $c_{s,a}>1$
   such that
   \begin{eqnarray*}
      \frac{1}{c_{s,a}}\leq \frac{u(t,x,y)}{\varphi_1(x)\varphi_1(y)}\leq c_{s,a}
   \end{eqnarray*}
   for every $(t,|x|,y)\in (s-\varepsilon_s, s+\varepsilon_s)\times(x_{s,a},\infty)\times(a-\varepsilon_a, a+\varepsilon_a)$. The family
   \begin{equation*}
      \left\{(s-\varepsilon_s, s+\varepsilon_s)\times(a-\varepsilon_a, a+\varepsilon_a)\right\}_{(s,a)\in[1,t_0]\times[-y_0,y_0]}
    \end{equation*}
    is an open cover of the compact set $[1,t_0]\times[-y_0,y_0]$. Therefore there exists a finite subcover
    \begin{equation*}
      \left\{(s_k-\varepsilon_{s_k}, s_k+\varepsilon_{s_k})\times(a_k-\varepsilon_{a_k}, a_k+\varepsilon_{a_k})\right\}_{k=1,2,\ldots,n}
    \end{equation*}
    of the set $[1,t_0]\times[-y_0,y_0]$. Putting $c=\max\{c_{s_k,a_k}:k=1,\ldots,n\}$, $x_1=\max\{x_{s_k,a_k}:k=1,\ldots,n\}$ we get that
   \begin{eqnarray*}
      \frac{1}{c}\leq \frac{u(t,x,y)}{\varphi_1(x)\varphi_1(y)}\leq c
   \end{eqnarray*}
   for every $(t,|x|,y)\in [1,t_0]\times(x_1,\infty)\times[-y_0,y_0]$. Due to the symmetry of $u(t,x,y)$ we get the analogous result for $(t,x,|y|)\in[1,t_0]\times[-x_0,x_0]\times[y_1,\infty)$
   Since $u$ and $\varphi_1$ are continuous and strictly positive we get
   \begin{eqnarray*}
      u(t,x,y)\geq c e^{-\lambda_1 t}\varphi_1(x)\varphi_1(y)
   \end{eqnarray*}
   for $(t,x,y)\in[1,t_0]\times[-\max\{x_0,x_1\},\max\{x_0,x_1\}]\times[-\max\{y_0,y_1\},\max\{y_0,y_1\}]$. This completes the proof.
\end{proof}

\section{Appendix: Airy functions}
For the convenience of the reader we summarize some basic properties of Airy functions used in this paper.

The Airy functions $\Ai(x)$ and $\Bi(x)$ are defined as two independent solutions of the Airy equation
\begin{equation}
\label{AiryEq}
y''-xy = 0\/,\quad x\in\R\/.
\end{equation}
The equation can be easily reduced to the Bessel equation (for $x\leq 0$) and to the modified Bessel equation
($x>0$). This allows to express the Airy functions in terms of Bessel functions $J_\vartheta$ and modified Bessel
functions $K_\vartheta$, $I_\vartheta$ in the following way:
\begin{eqnarray}
\Ai(x)
&=&
\left\{
  \begin{array}{cc}
     \dfrac{\sqrt{-x}}{3}\[J_{1/3}\(\dfrac23 (-x)^{3/2}\)+J_{-1/3}\(\dfrac23 (-x)^{3/2}\)\]\/, &x\leq0\/,\\
     \dfrac{1}{\pi} \sqrt{\dfrac{x}{3}}K_{1/3}\(\dfrac23 x^{3/2}\)\/,& x>0\/.
  \end{array}
  \right.
\end{eqnarray}
\begin{eqnarray}
\label{BiDefn}
\Bi(x)
&=& \left\{
  \begin{array}{cc}
     \sqrt{\dfrac{-x}{3}}\[J_{-1/3}\(\dfrac23 (-x)^{3/2}\)-J_{1/3}\(\dfrac23 (-x)^{3/2}\)\]\/, &x\leq0\/,\\
     \sqrt{\dfrac{x}{3}} \[I_{1/3}\(\dfrac23 x^{3/2}\)+I_{-1/3}\(\dfrac23 x^{3/2}\) \]\/,& x>0\/.
  \end{array}
  \right.
\end{eqnarray}
Using the relation $\Ai''(x) = x\Ai(x)$ we get that the $n$-th derivative of $\Ai$ is given by
\begin{equation}
\label{AiDerivatives}
   \Ai^{(n)}(x) = p_n(x)\Ai(x) + q_n(x)\Ai'(x)\/,
\end{equation}
where $p_n$ and $q_n$ are $n$th order polynomials defined by the recursive relations
\begin{eqnarray*}
   p_{n+1}(x) &=& p_n'(x)+xq_n(x)\/,\\
   q_{n+1}(x) &=& p_n(x) + q_n'(x)
\end{eqnarray*}
and $p_0(x)\equiv 1$, $q_0(x)\equiv 0$. Below we give formulae for $p_n$ and $q_n$ for $n=1,\ldots,10$.
\begin{center}
 $$  \begin{array}{rclcrcl}
      p_1(x)&=&0            &&q_1(x)&=&1\\
      p_2(x)&=&x            &&q_2(x)&=&0\\
      p_3(x)&=&1            &&q_3(x)&=&x\\
      p_4(x)&=&x^2          &&q_4(x)&=&2\\
      p_5(x)&=&4x           &&q_5(x)&=&x^2\\
      p_6(x)&=&x^3+4        &&q_6(x)&=&6x\\
      p_7(x)&=&9x^2         &&q_7(x)&=&x^3+10\\
      p_8(x)&=&x^4+28x      &&q_8(x)&=&12x^2\\
      p_9(x)&=&16x^3+28     &&q_9(x)&=&x^4+52x\\
      p_{10}(x)&=&x^5+100x^2&&q_{10}(x)&=&20x^3+80\\
      \end{array}$$
\end{center}

We recall some asymptotic results related to Airy functions. The asymptotic behaviour of $\Ai$ for large
arguments is given by (\cite{AS},10.4.59 and 10.4.60)
\begin{eqnarray}
   \label{Airy:asymp:infty}
   \Ai(x) &\cong& \frac{1}{2\pi^{1/2}}x^{-1/4}e^{-\frac{2}{3}x^{3/2}}\/,\quad x\to\infty\/,\\
   \label{Airy:asymp:minusinfty}
   \Ai(-x)&=& \frac{1}{\pi^{1/2}}\frac{\sin\left(\frac{2}{3}x^{3/2}+\pi/4\right)}{x^{1/4}}+O(x^{-7/4})\/,\quad x\to \infty\/.
\end{eqnarray}
The corresponding formula for $\Ai'$ is given by (\cite{AS}, 10.4.61)
\begin{eqnarray}
 \label{AiryPrime:asymp:infty}
 \Ai'(x) &\cong& -\frac{1}{2\pi^{1/2}}x^{1/4}e^{-\frac{2}{3}x^{3/2}}\/,\quad x\to\infty\/.
\end{eqnarray}
Asymptotic formulae for $\Ai(a_k')$ and $\Ai'(a_k)$ are (\cite{AS} 10.4.96 and 10.4.97)
\begin{eqnarray}
   \label{Airy:asymp:akprime}
   \Ai(a_k') &\sim& (-1)^{k-1} \pi^{-1/2}\left(\frac{3\pi}{2}\right)^{-1/6} (k-3/4)^{-1/6}\/,\quad k\to \infty\/,\\
   \label{Airy:asymp:ak}
   \Ai'(a_k) &\sim& (-1)^{k-1}\pi^{-1/2}\left(\frac{3\pi}{2}\right)^{1/6} (k-1/2)^{1/6}\/,\quad k\to \infty\/.
\end{eqnarray}

Finally we prove a result we have used in the previous section.
\begin{lem}
\label{UB:lemma}
  For any open set $D\subset\R$ and $a>0$ let $u\mapsto f(x,u)$ be a continuous function for every fixed $x\in D$
  such that $\int_a^y f(x,u)du$ is uniformly bounded for $(x,y)\in D\times [a,\infty)$. Assume that $(x,u)\mapsto
  g(x,u)$ is a uniformly bounded function on $D\times [a,\infty)$ such that for every $x\in D$ and $y>a$ the
  function $u\rightarrow g(x,u)$ has continuous derivative on $[a,y)$, and its derivative has at most one zero in
  $[a,y)$. Then
  \begin{equation*}
     F(x,y) = \int_a^y f(x,u)g(x,u)du
  \end{equation*}
  is a uniformly bounded function in $D\times [a,\infty)$.
\end{lem}
\begin{proof}
   Denote by $b$ the zero of the function $g(x,\cdot)$ in the interval $[a,y)$. If there are no zeroes of $g(x,\cdot)$
   in this interval $b$ can be arbitrarily chosen in the interval. Fix $x\in D$. Applying the second mean value theorem
   for integration to the intervals $[a,b]$ and $[b,y]$ apart it follows that there exist constants $\lambda_1\in[a,b]$
   and $\lambda_2\in[b,y]$ such that $F(x,y)$ is equal to
   \begin{eqnarray*}
     g(x,a)\int_a^{\lambda_1}f(x,u)du+g(x,b)\int_{\lambda_1}^bf(x,u)du+g(x,b)\int_b^{\lambda_2}f(x,u)du+
     g(x,y)\int_{\lambda_2}^yf(x,u)du\/.
   \end{eqnarray*}
   Using the assumption that the functions $g(x,u)$ and $\int_a^y f(x,u)du$ are uniformly bounded, the result follows.
\end{proof}

\bigskip
\noindent
\textbf{Acknowledgments:} It is a pleasure to thank T. Byczkowski, T. Kulczycki and A. Strohmaier for
discussions.

\end{document}